\documentclass[10pt]{article}

\usepackage{amsfonts}
\usepackage{amsmath}
\usepackage{amssymb}
\usepackage{amsthm}
\usepackage{algorithm}
\usepackage{algorithmicx}
\usepackage{algpseudocode}
\usepackage{graphicx}
\usepackage{indentfirst}
\usepackage{placeins}
\usepackage{flafter}
\usepackage[hidelinks]{hyperref}
\usepackage{url}

\renewcommand{\baselinestretch}{1.2}
\numberwithin{equation}{section}
\theoremstyle{plain}
\newtheorem{theorem}{Theorem}[section]
\newtheorem{lemma}[theorem]{Lemma}
\newtheorem{proposition}[theorem]{Proposition}
\newtheorem{corollary}[theorem]{Corollary}
\theoremstyle{remark}
\newtheorem{remark}[theorem]{Remark}
\numberwithin{algorithm}{section}
\algrenewcommand\algorithmicrequire{\textbf{Input:}}
\algrenewcommand\algorithmicensure{\textbf{Output:}}

\newcommand{\R}{\mathbb{R}}
\newcommand{\Mle}{\mathcal{M}_{\le r}}
\newcommand{\eps}{\varepsilon}
\newcommand{\rank}{\operatorname{rank}}
\newcommand{\range}{\operatorname{range}}

\hypersetup{
 pdftitle={Robust High-Order Projector-Splitting Integrators},
 pdfauthor={Shiheng Zhang and Jingwei Hu},
 pdfkeywords={Dynamical low-rank approximation; projector splitting; exactness; approximate increment; robust low-rank integrator; small singular values; Runge--Kutta method}
}

\begin{document}

\title{Robust High-Order Projector-Splitting Integrators}
\author{Shiheng Zhang\thanks{Department of Applied Mathematics, University of Washington, Seattle, WA 98195, USA (\texttt{shzhang3@uw.edu}). Corresponding author.}
\and
Jingwei Hu\thanks{Department of Applied Mathematics, University of Washington, Seattle, WA 98195, USA (\texttt{hujw@uw.edu}).}}
%\date{}

\maketitle

\begin{abstract}
We develop and analyze an increment-based \(K\)-\(L\) integrator for dynamical low-rank approximation. This formulation eliminates the separate backward \(S\)-step and is algebraically equivalent to the fixed-increment projector-splitting integrator of Lubich and Oseledets \cite{LubichOseledets2014}. We establish a relaxed-exactness estimate that bounds the error in the low-rank approximation in terms of the prescribed increment error, with constants independent of small singular values. This estimate provides a framework for achieving robust high-order convergence through sufficiently accurate increments. We then construct low-rank Runge--Kutta methods in which both the stage and final approximations are computed using the $K$-$L$ integrator, without basis augmentation or rank truncation. Since classical Runge--Kutta order conditions alone do not guarantee robust high-order convergence for this construction, we derive sufficient conditions on the coefficients that ensure the required stage accuracy and establish robust convergence of order $p$. We identify explicit Runge--Kutta schemes of orders $p=2,3,4$ satisfying these conditions, and numerical experiments confirm the theoretical results.
\end{abstract}



%%%%%%%%%%%%%%
\section{Introduction}
\label{sec:introduction}

Dynamical low-rank approximation of time-dependent matrices \cite{KochLubich2007} approximates the solution \(A(t)\in\R^{m\times n}\) of a large matrix differential equation
\begin{equation}\label{eq:full-ode}
 \dot A(t)=F(t,A(t)),
\end{equation}
by a rank-$r$ matrix \(Y(t)\), represented in the factorized form
\begin{equation}\label{eq:factorization}
 Y=USV^\top,
\end{equation}
where \(U\in\R^{m\times r}\) and \(V\in\R^{n\times r}\) have orthonormal columns, and \(S\in\R^{r\times r}\) is nonsingular. The rank constraint is imposed by projecting the right-hand side of \eqref{eq:full-ode} onto the tangent space of the rank-\(r\) manifold,
\begin{equation}\label{eq:projected-ode}
 \dot Y(t)=P_r(Y(t))F(t,Y(t)).
\end{equation}
At \(Y=USV^\top\), the orthogonal projection onto the tangent space is the alternating sum
\begin{equation}\label{eq:tangent-projector}
 P_r(Y)Z=ZVV^\top-UU^\top ZVV^\top+UU^\top Z,
 \qquad Z\in\R^{m\times n}.
\end{equation}
Throughout the paper, \(\|\cdot\|\) denotes the Frobenius norm and \(\|\cdot\|_2\) the spectral norm.

Equation \eqref{eq:projected-ode} can be written as a system of differential equations for \(U\), \(S\), and \(V\), where the equations for \(U\) and \(V\) involve \(S^{-1}\) \cite{KochLubich2007}.
Standard integrators applied directly to this system, such as explicit Runge--Kutta methods, may therefore be subject to severe step-size restrictions when \(S\) has small singular values.

The projector-splitting integrator of Lubich and Oseledets \cite{LubichOseledets2014} avoids this difficulty by a Lie--Trotter splitting according to the three terms of \eqref{eq:tangent-projector}. For a general matrix differential equation \eqref{eq:full-ode}, this leads to the three subproblems
\begin{equation}\label{eq:psi-subproblems}
 \dot Y=F(t,Y)VV^\top,\qquad
 \dot Y=-UU^\top F(t,Y)VV^\top,\qquad
 \dot Y=UU^\top F(t,Y).
\end{equation}
These subproblems are solved sequentially, each from \(t_0\) to \(t_1=t_0+h\), starting from the result of the previous one. The first subproblem (often called the $K$-step with $K=US$) evolves \(U\) and \(S\) while keeping \(V\) fixed, the second ($S$-step) evolves \(S\) while keeping \(U\) and \(V\) fixed, and the third ($L$-step with $L=VS^\top$) evolves \(V\) and \(S\) while keeping \(U\) fixed. Each subproblem therefore reduces to a differential equation for a single factor without involving \(S^{-1}\). Under the assumptions below, the resulting integrator admits error bounds independent of small singular values.

To state this robustness, let \(\delta\) bound the initial error \(\|Y_0-A(t_0)\|\), and let \(\eps_r\) bound the normal component of \(F\) in the Frobenius norm, that is, \(\|(I-P_r(Y))F(t,Y)\|\le\eps_r\) for \(t_0\le t\le T\) and all rank-\(r\) matrices \(Y\). We call a low-rank integrator \emph{robust of order \(p\)} if its approximations \(Y_k\) at \(t_k=t_0+kh\) satisfy
\begin{equation}\label{eq:robust-order-p}
 \max_{t_0\leq t_k\le T}\|Y_k-A(t_k)\|
 \le C\bigl(\delta+\eps_r+h^p\bigr),
 \qquad 0<h\le h_0,
\end{equation}
where \(C\) and \(h_0\) are independent of \(\delta\), \(\eps_r\), and the singular values of the exact solution \(A(t)\) and of the numerical approximations. For bounded and Lipschitz continuous \(F\), Kieri, Lubich, and Walach \cite{KieriLubichWalach2016} proved that the Lie--Trotter projector-splitting integrator and its Strang composition, with exact integration of the subproblems, satisfy \eqref{eq:robust-order-p} with $p=1$, i.e., both are robust of order one. Although the Strang composition has classical order two \cite{LubichOseledets2014}, a counterexample in \cite{Zhang2026Counterexample} shows that it is not, in general, robust of order two.

The middle term of \eqref{eq:tangent-projector} enters with a negative sign, leading to a backward $S$-step that may be unstable for strongly dissipative problems \cite{CerutiLubich2022,ZhangHu2026Stability}. To avoid this backward substep, we consider an alternative decomposition of the tangent-space projector by grouping the last two terms of \eqref{eq:tangent-projector}:
\begin{equation}\label{eq:grouped-projector}
 P_r(Y)Z=ZVV^\top+UU^\top Z(I-VV^\top).
\end{equation}
The first term projects the rows of $Z$ onto the subspace spanned by \(V\), while the second projects the columns of the remaining component $Z(I-VV^\top)$ onto the subspace spanned by $U$. This decomposition motivates a two-step projector-splitting formulation without a separate backward $S$-step, which we call the \(K\)-\(L\) integrator. In this paper, we focus on an increment-based realization of the \(K\)-\(L\) integrator. We show in Section~\ref{sec:psi} that, for any prescribed matrix increment \(\Delta A\), this realization is algebraically equivalent to the \(K\)-\(S\)-\(L\) integrator \cite{LubichOseledets2014} when the same increment is used in all substeps. 

Our approach to achieving robust high order is to improve the prescribed increment rather than increase the order of the projector splitting. The key question is whether errors in the prescribed increment lead to comparably small errors in the $K$-$L$ output, uniformly with respect to small singular values. Motivated by the exactness property of the $K$-$S$-$L$ integrator, we investigate whether this property can be relaxed to accommodate approximate increments. In particular, we consider a target matrix $A_*$ satisfying the range condition \(\operatorname{range}(A_*)=\operatorname{range}(A_*V_0)\), where $V_0$ is the initial row basis, and study how increment approximation errors affect the $K$-$L$ output.

Our first result establishes the following estimate under the range condition (Theorem~\ref{thm:stable-exactness}):
\begin{equation}\label{eq:relaxed-exactness-intro}
 \|Y_1-A_*\|
 \le C\|Y_0+\Delta A-A_*\|.
\end{equation}
This estimate shows that the error in the $K$-$L$ output is controlled by the discrepancy between the prescribed increment $\Delta A$ and the target increment $A_*-Y_0$. The constant \(C\) depends on the alignment between the row subspace of \(A_*\) and the subspace spanned by \(V_0\), rather than on the magnitudes of the nonzero singular values of \(A_*\). We call this property \emph{relaxed exactness}.

The range condition restricts the target matrix $A_*$ to rank at most \(r\), whereas the exact solution of \eqref{eq:full-ode} may have higher rank. To overcome this restriction, we construct a nearby comparison matrix satisfying the range condition, with an alignment constant bounded independently of small singular values. This yields a local error estimate consisting of the increment approximation error and an additional \(O(h\eps_r)\) term (Theorem~\ref{thm:accuracy-transfer}). Applying this estimate at each step yields the global error bound \eqref{eq:robust-order-p}, provided the increment errors are uniformly \(O(h^{p+1}+h\eps_r)\) (Theorem~\ref{thm:general-increments}).

To construct increments with the required accuracy while maintaining low-rank approximations, we use explicit Runge--Kutta schemes and compute both the stage and final approximations with the $K$-$L$ integrator, without basis augmentation or rank truncation. Since this modifies the internal stages, classical Runge--Kutta order conditions alone do not guarantee robust high-order convergence. Theorem~\ref{thm:rk-stage-criterion} provides sufficient conditions on the Runge--Kutta coefficients that ensure the required stage accuracy and hence robust order \(p\). We verify these conditions for forward Euler, the explicit midpoint method, Heun's third-order method, and Merson's fourth-order method, obtaining robust orders \(p=1,2,3,4\), respectively.

{\bf Related work.} The XL integrator \cite{einkemmer2025asymptotic} also performs a \(K\)-step followed by an \(L\)-step, but uses basis augmentation between the two steps. In \cite{HochbruckNeherSchrammer2023}, an increment computed using the classical fourth-order Runge--Kutta method for the full matrix equation is used in a fixed-increment projector-splitting integrator. Its internal Runge--Kutta stages are not constrained to rank $r$, so evaluations of $F$ at higher-rank matrices may be required. Our construction instead computes both the stage and final approximations at rank at most $r$. Robust high-order error bounds have also been established for other classes of dynamical low-rank integrators. Projected Runge--Kutta methods \cite{KieriVandereycken2019} combine tangent-space projections of $F$ with quasi-optimal low-rank approximations, such as truncated singular value decompositions. Their analysis likewise relates robust convergence order to internal stage accuracy. The fixed-rank and rank-adaptive basis-update and Galerkin (BUG) integrators \cite{CerutiLubich2022,CerutiKuschLubich2022} admit robust first-order error bounds, while higher-order extensions include the midpoint \cite{CerutiEinkemmerKuschLubich2024}, parallel \cite{Kusch2025}, and Runge–Kutta \cite{NobileRiffaud2026} constructions. These high-order BUG methods use basis augmentation followed by rank truncation. In contrast, the methods developed here achieve robust high-order convergence under suitable Runge--Kutta coefficient conditions, without basis augmentation or rank truncation.

{\bf Paper organization.} The remainder of the paper is organized as follows. Section~\ref{sec:psi} presents the $K$-$L$ integrator and its equivalence to the fixed-increment $K$-$S$-$L$ integrator. Sections~\ref{sec:increments} and \ref{sec:global} establish the relaxed-exactness estimate and the local and global error bounds. Section~\ref{sec:rk} develops the Runge--Kutta constructions and gives sufficient conditions for robust high-order convergence. Section~\ref{sec:numerical} presents numerical experiments, and Section~\ref{sec:conclusion} concludes the paper.

%%%%%%%%%%%%%%%%%%%%%%%%
\section{The \(K\)-\(L\) integrator}
\label{sec:psi}

In this section, we first recall the increment-based \(K\)-\(S\)-\(L\) projector-splitting integrator from \cite{LubichOseledets2014}. Eliminating its intermediate $S$-step yields the two-step \(K\)-\(L\) integrator used throughout the paper.

Let \(Y_0=U_0S_0V_0^\top\), where \(U_0\) and \(V_0\) have orthonormal columns. For the numerical formulations below, $S_0$ may be singular, so that $\rank(Y_0)\leq r$. Let \(\Delta A\in\R^{m\times n}\) be a prescribed increment for the step from \(t_0\) to \(t_1=t_0+h\). Here $(Q,R)=\mathrm{QR}(X)$ denotes a QR decomposition $X=QR$ with $r$ orthonormal columns in $Q$; if $X$ is rank deficient, an orthonormal completion is used. One step of the fixed-increment \(K\)-\(S\)-\(L\) integrator consists of the following three substeps:
\begin{subequations}\label{eq:ksl}
\begin{alignat}{3}
 &\text{\(K\)-step:}\quad & K_1&=U_0S_0+\Delta A\,V_0,
  &\qquad (U_1,\widehat S_1)&=\mathrm{QR}(K_1),
 \label{eq:ksl-k}\\
 &\text{\(S\)-step:}\quad & \widetilde S_0&=\widehat S_1-U_1^\top\Delta A\,V_0,
 \label{eq:ksl-s}\\
 &\text{\(L\)-step:}\quad & L_1&=V_0\widetilde S_0^\top+\Delta A^\top U_1,
  &\qquad (V_1,S_1^\top)&=\mathrm{QR}(L_1),
 \label{eq:ksl-l}
\end{alignat}
\end{subequations}
and the approximation at \(t_1\) is given by \(Y_1=U_1S_1V_1^\top\). The same increment $\Delta A$ is used in all three substeps. It may be the exact increment \(\Delta A=A(t_1)-A(t_0)\) of a solution $A(t)$ to \eqref{eq:full-ode}, or an approximation such as the forward Euler increment \(\Delta A=hF(t_0,Y_0)\). For a general state-dependent right-hand side $F(t,Y)$, however, this fixed-increment formulation differs from the projector-splitting integrator with exactly solved subproblems \cite{KieriLubichWalach2016}, whose right-hand sides are evaluated along the respective subproblem solutions.

The \(K\)-\(L\) integrator uses the same prescribed increment in the following two substeps:
\begin{subequations}\label{eq:reduced-kl}
\begin{alignat}{3}
 &\text{\(K\)-step:}\quad & K_1&=U_0S_0+\Delta A\,V_0=(Y_0+\Delta A)V_0,
  &\qquad (U_1,\widehat S_1)&=\mathrm{QR}(K_1),
 \label{eq:two-factor-k}\\
 &\text{\(L\)-step:}\quad & L_1&=V_0S_0^\top U_0^\top U_1+\Delta A^\top U_1=(Y_0+\Delta A)^\top U_1,
  &\qquad (V_1,S_1^\top)&=\mathrm{QR}(L_1),
 \label{eq:two-factor-l}
\end{alignat}
\end{subequations}
and \(Y_1=U_1S_1V_1^\top\).  

Since \(Y_1=U_1L_1^\top\) and \(K_1=U_1\widehat S_1\), we obtain from \eqref{eq:reduced-kl} that 
\begin{equation}\label{eq:ksl-old-row-identities}
 Y_1=U_1U_1^\top(Y_0+\Delta A),
 \qquad U_1U_1^\top(Y_0+\Delta A)V_0=(Y_0+\Delta A)V_0.
\end{equation}
These identities are the basis of the relaxed-exactness estimate in Section~\ref{sec:increments}.

The following proposition establishes the equivalence of \eqref{eq:ksl} and \eqref{eq:reduced-kl}.
\begin{proposition}
\label{prop:reduced-kl}
Given the same input $Y_0=U_0 S_0V_0^{\top}$ and the same prescribed increment \(\Delta A\), the \(K\)-\(S\)-\(L\) integrator \eqref{eq:ksl} and the \(K\)-\(L\) integrator \eqref{eq:reduced-kl} produce the same output  \(Y_1=U_1S_1V_1^\top\), provided they use the same orthonormal factor \(U_1\) in their identical \(K\)-substeps.
\end{proposition}

\begin{proof}
The \(K\)-steps \eqref{eq:ksl-k} and \eqref{eq:two-factor-k} are identical, so they produce the same \(K_1\). We use the same orthonormal factor \(U_1\) in both formulations. Then \(K_1=U_1\widehat S_1\) gives \(\widehat S_1=U_1^\top K_1\), and \eqref{eq:ksl-s} becomes
\[
 \widetilde S_0=U_1^\top K_1-U_1^\top\Delta A\,V_0=
 U_1^\top(U_0S_0+\Delta A\,V_0)-U_1^\top\Delta A\,V_0
 =U_1^\top U_0S_0.
\]
Substituting this into  \eqref{eq:ksl-l} yields
\[
 L_1=V_0(U_1^\top U_0S_0)^\top+\Delta A^\top U_1
 =V_0S_0^\top U_0^\top U_1+\Delta A^\top U_1,
\]
which is the same \(L_1\) as in \eqref{eq:two-factor-l}. Therefore, both integrators produce the same output $Y_1$.
\end{proof}

%%%%%%%%%%%%%%%%%%%%%%%%
\section{Relaxed exactness and the local error}
\label{sec:increments}

In this section, we first establish an algebraic error estimate for the \(K\)-\(L\) integrator with a prescribed increment. We then use this estimate to derive a local error bound for the matrix differential equation \eqref{eq:full-ode}, with constants independent of small singular values.

\subsection{Relaxed exactness}
\label{sec:stable-exactness}

The following theorem extends the exactness property of the projector-splitting integrator \cite{LubichOseledets2014} to approximate increments. We formulate the result for a target matrix \(A_*\in\R^{m\times n}\) satisfying the range condition \(\range(A_*)=\range(A_*V_0)\), or, equivalently,
\begin{equation}\label{eq:endpoint-range-condition}
 A_*=(A_*V_0)Q
 \qquad\text{for some }Q\in\R^{r\times n}.
\end{equation}

\begin{theorem}
\label{thm:stable-exactness}
Let \(Y_1\) be computed using the \(K\)-\(L\) integrator \eqref{eq:reduced-kl} with initial approximation $Y_0$ and a prescribed increment \(\Delta A\), and let \(A_*\) satisfy \eqref{eq:endpoint-range-condition}. Then, for every \(Q\) satisfying \eqref{eq:endpoint-range-condition},
\begin{equation}\label{eq:ksl-stability}
 \|Y_1-A_*\|\le\max\{1,\|Q\|_2\}\,\|Y_0+\Delta A-A_*\|.
\end{equation}
\end{theorem}

\begin{proof}
Set \(P:=U_1U_1^\top\) and \(E:=Y_0+\Delta A-A_*\). By \eqref{eq:ksl-old-row-identities}, \(Y_1=P(Y_0+\Delta A)\) and \((I-P)(Y_0+\Delta A)V_0=0\). Hence,
\begin{align}
 Y_1-A_*
 &=P(Y_0+\Delta A)-P A_*-(I-P)A_*=PE-(I-P)A_*,
   \label{eq:ksl-proof-1}\\
 (I-P)A_*V_0
 &=-(I-P)EV_0.
   \label{eq:ksl-proof-2}
\end{align}
By \eqref{eq:endpoint-range-condition} and \eqref{eq:ksl-proof-2},
\[
 (I-P)A_*
 =(I-P)A_*V_0Q
 =-(I-P)EV_0Q,
\]
and \eqref{eq:ksl-proof-1} becomes
\[
 Y_1-A_*
 =PE
 +(I-P)EV_0Q.
\]
The matrices \(P\) and \(I-P\) are the orthogonal projectors onto \(\range(U_1)\) and its orthogonal complement, respectively. Since \(P(I-P)=0\), the two terms in the preceding
expression are orthogonal in the Frobenius inner product. Moreover,
\[
 \|P Z\|^2+\|(I-P)Z\|^2=\|Z\|^2,
 \qquad Z\in\R^{m\times n}.
\]
Using these identities and \(\|V_0\|_2=1\), we obtain
\begin{align*}
 \|Y_1-A_*\|^2
 &=\|PE\|^2
   +\|(I-P)EV_0Q\|^2\\
 &\le\|PE\|^2
   +\|Q\|_2^2\|(I-P)E\|^2\\
 &\le\max\{1,\|Q\|_2^2\}\,\|E\|^2.
\end{align*}
Taking square roots proves \eqref{eq:ksl-stability}.
\end{proof}

Theorem~\ref{thm:stable-exactness} gives a relaxed-exactness estimate with \(C=\max\{1,\|Q\|_2\}\). In particular, if the prescribed increment is exact relative to the target matrix, that is, \(\Delta A=A_*-Y_0\), then \(Y_1=A_*\), recovering the classical exactness property of the projector-splitting integrator \cite{LubichOseledets2014} under the stated range condition.

\subsection{Standing assumptions}
\label{sec:assumptions}

We now analyze the error of the \(K\)-\(L\) integrator for \eqref{eq:full-ode}. Let \(1\le r\le\min(m,n)\). Fix \(0\le t_0<T\), and set
\begin{equation}\label{eq:rank-strata}
 \Mle:=\{Y\in\R^{m\times n}:\rank(Y)\le r\}.
\end{equation}
We assume that \(F:[0,T]\times\R^{m\times n}\to\R^{m\times n}\) is continuous and satisfies the following conditions for all \(0\le t\le T\).

\begin{enumerate}
\item \(F\) is bounded and Lipschitz continuous in its second argument, with constants $B>0$ and \(L_F>0\): for all \(X,Z\in\R^{m\times n}\),
\begin{equation}\label{eq:assumption-BL}
 \|F(t,X)\|\le B,
 \qquad
 \|F(t,X)-F(t,Z)\|\le L_F\|X-Z\|.
\end{equation}

\item The normal component of \(F\) is bounded in the spectral norm by \(\eps_r\ge0\): for every \(X\in\Mle\),
\begin{equation}\label{eq:assumption-normal}
 \|(I-XX^+)F(t,X)(I-X^+X)\|_2\le\eps_r,
\end{equation}
where \(X^+\) is the Moore--Penrose pseudoinverse, and \(XX^+\) and \(X^+X\) are the orthogonal projectors onto \(\range(X)\) and \(\range(X^\top)\), respectively.
\end{enumerate}

For matrices $X$ of rank \(r\), the expression in \eqref{eq:assumption-normal} coincides with the normal component \((I-P_r(X))F(t,X)\). The pseudoinverse formulation extends this condition to  matrices of rank less than $r$, where the tangent-space projector $P_r(X)$ is not defined.

\subsection{The local error}
\label{sec:local-error}

For the local analysis, let \(A(t)\) denote the exact solution of \eqref{eq:full-ode} with initial value \(A(t_0)=Y_0\), and let \(t_1=t_0+h\). If \(\rank(A(t_1))>r\), then \(A(t_1)\) cannot satisfy the range condition \eqref{eq:endpoint-range-condition}. Thus, Theorem~\ref{thm:stable-exactness} cannot be applied directly with \(A_*=A(t_1)\). To overcome this difficulty, we construct a comparison matrix \(\widetilde A_1\in\Mle\) that is close to \(A(t_1)\) and satisfies \(\widetilde A_1=(\widetilde A_1V_0)Q_h\), with \(\|Q_h\|_2\) bounded independently of the singular values of \(Y_0\).

To state the bounds, define
\begin{equation}\label{eq:local-error-functions}
 \varphi(t):=\frac{e^{L_Ft}-1}{L_F},
 \qquad
 \gamma(t):=\frac{1}{\sqrt{1-rL_F^2t^2}},
\end{equation}
where \(\varphi\) is defined for \(t\ge0\) and \(\gamma\) for \(0\le t<1/(\sqrt r\,L_F)\). We assume \(0<h\le h_*\), where
\begin{equation}\label{eq:comparison-hstar}
 h_*:=\frac{1}{\sqrt{2r}\,L_F},\qquad \gamma_*:=\gamma(h_*)=\sqrt{2}.
\end{equation}
Then \(1<\gamma(h)\le\gamma_*\) for \(0<h\le h_*\).

\begin{lemma}
\label{lem:comparison-matrix}
Under the standing assumptions, let \(Y_0\in\Mle\), let \(V_0\in\R^{n\times r}\) satisfy \(V_0^\top V_0=I_r\) and \(Y_0=Y_0V_0V_0^\top\), and let \(A\) solve \eqref{eq:full-ode} with
\(A(t_0)=Y_0\). For \(0<h\le h_*\) and \(t_1=t_0+h\le T\), there exist matrices \(\widetilde A_1\in\Mle\) and \(Q_h\in\R^{r\times n}\) such that
\begin{equation}\label{eq:comparison-matrix-error}
 \|\widetilde A_1-A(t_1)\|
 \le\sqrt{\min(m,n)}\,\eps_r\varphi(h),
\end{equation}
and
\begin{equation}\label{eq:comparison-matrix-factorization}
 \widetilde A_1=(\widetilde A_1V_0)Q_h,
 \qquad Q_hV_0=I_r,
 \qquad \|Q_h\|_2\le\gamma(h)\le\gamma_*.
\end{equation}
The bounds are independent of the singular values of \(Y_0\).
\end{lemma}

The proof is given in Appendix~\ref{app:comparison-matrix}. The comparison matrix is introduced solely for the analysis and is not computed by the $K$-$L$ integrator. Since this matrix satisfies the range condition \eqref{eq:endpoint-range-condition}, we can apply Theorem~\ref{thm:stable-exactness} directly to obtain the following robust local error estimate.

\begin{theorem} 
\label{thm:accuracy-transfer}
Under the standing assumptions, let $A$ solve \eqref{eq:full-ode} with $A(t_0)=Y_0=U_0S_0V_0^\top$. Let $0<h\le h_*$, with $h_*$ defined in \eqref{eq:comparison-hstar}, and $t_1=t_0+h\le T$. For any prescribed increment $\Delta A$, the output $Y_1$ of the K-L integrator \eqref{eq:reduced-kl} satisfies
\begin{equation}
\|Y_1-A(t_1)\|
\le \gamma(h)\|Y_0+\Delta A-A(t_1)\| +(1+\gamma(h))\sqrt{\min(m,n)}
\,\varepsilon_r\varphi(h).
\label{eq:local-error}
\end{equation}
\end{theorem}

\begin{proof}
Let $\widetilde A_1$ and $Q_h$ be given by Lemma~\ref{lem:comparison-matrix}. By \eqref{eq:comparison-matrix-factorization},
\[
\widetilde A_1=(\widetilde A_1V_0)Q_h,
\qquad
\|Q_h\|_2\le\gamma(h).
\]
Thus, $\widetilde A_1$ satisfies the range condition \eqref{eq:endpoint-range-condition}. Applying Theorem~\ref{thm:stable-exactness} with $A_*=\widetilde A_1$ gives
\[
\begin{aligned}
\|Y_1-\widetilde A_1\|
&\le \max\{1,\|Q_h\|_2\}
     \|Y_0+\Delta A-\widetilde A_1\|\\
&\le \gamma(h)\|Y_0+\Delta A-\widetilde A_1\|,
\end{aligned}
\]
where we used $\gamma(h)\ge 1$.

By the triangle inequality,
\[
\begin{aligned}
\|Y_1-A(t_1)\|
&\le \|Y_1-\widetilde A_1\|
     +\|\widetilde A_1-A(t_1)\|\\
&\le \gamma(h)\|Y_0+\Delta A-\widetilde A_1\|
     +\|\widetilde A_1-A(t_1)\|\\
&\le \gamma(h)\|Y_0+\Delta A-A(t_1)\| +(1+\gamma(h))
     \|\widetilde A_1-A(t_1)\|.
\end{aligned}
\]
Applying \eqref{eq:comparison-matrix-error} yields \eqref{eq:local-error}.
\end{proof}

Since $\gamma(h)=1+O(h^2)$, the first term in \eqref{eq:local-error} measures the increment approximation error $\|Y_0+\Delta A-A(t_1)\|$, multiplied by a factor close to one. The second term is $O(\sqrt{\min(m,n)}\,h\varepsilon_r)$, since $1+\gamma(h)$ is uniformly bounded and $\varphi(h)\le he^{L_Fh}$. The constants in these estimates depend only on $r$ and $L_F$ and are independent of the singular values.

%%%%%%%%%%%%%%%%%%%%%%%%
\section{The global error}
\label{sec:global}

We now extend the local error estimate of Theorem~\ref{thm:accuracy-transfer} to obtain a global error bound for the $K$-$L$ integrator. The analysis relates the error at each step to the accuracy of the prescribed increment and the error accumulated from previous steps.

Let \(t_k=t_0+kh\), and suppose that the numerical approximation \(Y_k\approx A(t_k)\) is represented as
\begin{equation}\label{eq:method-start}
 Y_k=U_kS_kV_k^\top,
 \qquad U_k^\top U_k=V_k^\top V_k=I_r,
 \qquad S_k\in\R^{r\times r}.
\end{equation}
Here, \(A(t)\) denotes the exact solution of \eqref{eq:full-ode} with initial value $A(t_0)$. At time \(t_k\), \(Y_{k+1}\) is computed by applying the \(K\)-\(L\) integrator \eqref{eq:reduced-kl} to \((U_k,S_k,V_k)\) with the prescribed increment \(\Delta A_k\). To analyze the error of this step, let \(A_k(t)\) denote the exact solution of \eqref{eq:full-ode} initialized with \(Y_k\) at \(t_k\):
\begin{equation}\label{eq:increment-local-solution}
 \dot A_k(t)=F(t,A_k(t)),\qquad A_k(t_k)=Y_k,
\end{equation}
and define
\begin{equation}\label{eq:increment-defect}
 d_k:=\|Y_k+\Delta A_k-A_k(t_{k+1})\|,
 \qquad d_h:=\max_{t_0\leq t_{k+1}\le T}d_k.
\end{equation}
In the error bounds below, \(a=O(b)\) means \(a\le Cb\), where \(C\) does not depend on \(h\), \(k\), the initial error, \(\eps_r\), or the singular values of the exact and numerical solutions.

\begin{theorem}
\label{thm:general-increments}
Under the standing assumptions, let \(A(t)\) solve \eqref{eq:full-ode} with
\begin{equation}\label{eq:initial-error}
 \|Y_0-A(t_0)\|\le\delta,
\end{equation}
and let \(0<h\le h_*\). Then
\begin{align}
 \max_{t_0\leq t_k\le T}\|Y_k-A(t_k)\|
 &\le e^{L_FT}\delta+\varphi(T)
   \left(\gamma_*\frac{d_h}{h}+(1+\gamma_*)\sqrt{\min(m,n)}\,\eps_r\right)
 \label{eq:general-increment-explicit-bound}\\
 &=O\left(\delta+\eps_r+\frac{d_h}{h}\right).
 \label{eq:general-increment-global-error}
\end{align}
In particular, for integer \(p\ge1\),
\begin{equation}\label{eq:increment-order-transfer}
 d_h=O(h^{p+1}+h\eps_r)
 \quad\Longrightarrow\quad
 \max_{t_0\leq t_k\le T}\|Y_k-A(t_k)\|=O(\delta+\eps_r+h^p).
\end{equation}
\end{theorem}

\begin{proof}
Set \(e_k:=\|Y_k-A(t_k)\|\).  For \(t_k\le t\le t_{k+1}\), the integral forms of \eqref{eq:full-ode} and \eqref{eq:increment-local-solution} give
\[
 A_k(t)-A(t)=Y_k-A(t_k)   +\int_{t_k}^{t}\bigl(F(t',A_k(t'))-F(t',A(t'))\bigr)\,dt'.
\]
The Lipschitz bound in \eqref{eq:assumption-BL} and Gronwall's inequality yield
\[
 \|A_k(t)-A(t)\|
 \le e_k+L_F\int_{t_k}^{t}\|A_k(t')-A(t')\|\,dt'
 \quad\Longrightarrow\quad
 \|A_k(t)-A(t)\|\le e^{L_F(t-t_k)}e_k.
\]
Applying Theorem~\ref{thm:accuracy-transfer} to the solution $A_k(t)$ on \([t_k,t_{k+1}]\) and using the preceding bound at \(t=t_{k+1}\), we obtain
\begin{align}
 e_{k+1}
 &\le \|Y_{k+1}-A_k(t_{k+1})\|+\|A_k(t_{k+1})-A(t_{k+1})\| \notag\\
 &\le e^{L_Fh}e_k+\gamma(h)d_k
       +(1+\gamma(h))\sqrt{\min(m,n)}\,\eps_r\varphi(h).
 \label{eq:general-increment-recurrence}
\end{align}
Since \(e_0\le\delta\) by \eqref{eq:initial-error}, induction on \(k\) in \eqref{eq:general-increment-recurrence} gives
\begin{equation}\label{eq:general-increment-iteration}
 e_k\le e^{L_Fkh}\delta+
 \sum_{j=0}^{k-1}e^{L_F(k-1-j)h}
 \left(\gamma(h) d_j+(1+\gamma(h))\sqrt{\min(m,n)}\,\eps_r\varphi(h)\right).
\end{equation}
For \(t_k\le T\), each \(d_j\) in this sum satisfies \(d_j\le d_h\le\varphi(h)\,d_h/h\), since \(\varphi(h)\ge h\). Using \(\gamma(h)\le\gamma_*\) and the geometric sum \(\varphi(h)\sum_{j=0}^{k-1}e^{L_Fjh}=\varphi(kh)\), the sum in \eqref{eq:general-increment-iteration} satisfies
\[
 \begin{aligned}
 &\sum_{j=0}^{k-1}e^{L_F(k-1-j)h}
 \left(\gamma(h) d_j+(1+\gamma(h))\sqrt{\min(m,n)}\,\eps_r\varphi(h)\right)\\
 \le &\left(\gamma_* \frac{d_h}{h}+(1+\gamma_*)\sqrt{\min(m,n)}\,\eps_r\right)
   \varphi(h)\sum_{j=0}^{k-1}e^{L_Fjh}=\left(\gamma_* \frac{d_h}{h}+(1+\gamma_*)\sqrt{\min(m,n)}\,\eps_r\right)\varphi(kh).
 \end{aligned}
\]
Since \(kh\le T\), we have \(e^{L_Fkh}\le e^{L_FT}\) and \(\varphi(kh)\le\varphi(T)\), which proves \eqref{eq:general-increment-explicit-bound}, and hence \eqref{eq:general-increment-global-error}.
\end{proof}

%%%%%%%%%%%%%%%%%%%%%%%%
\section{Robust high-order low-rank Runge--Kutta integrators}
\label{sec:rk}

Theorem~\ref{thm:general-increments} reduces the task of achieving robust order $p$ to constructing increments with \(d_h=O(h^{p+1}+h\varepsilon_r)\). We use explicit Runge--Kutta schemes to construct such increments, with the $K$-$L$ integrator used to compute both the stage and final approximations, all of rank at most $r$. After introducing the method, we establish sufficient conditions on the Runge--Kutta coefficients for robust high-order convergence and identify explicit schemes satisfying these conditions.

\subsection{The \texorpdfstring{$K$-$L$}{K-L} Runge--Kutta method}
\label{sec:rk-criterion}

Consider an $s$-stage explicit Runge--Kutta method, $s\ge1$,
with Butcher tableau
\[
\begin{array}{c|ccccc}
0      & 0      & 0      & \cdots & 0      & 0\\
c_2    & a_{21} & 0      & \cdots & 0      & 0\\
c_3    & a_{31} & a_{32} & \ddots & \vdots & \vdots\\
\vdots & \vdots & \vdots & \ddots & 0      & 0\\
c_s    & a_{s1} & a_{s2} & \cdots & a_{s,s-1} & 0\\
\hline
       & b_1    & b_2    & \cdots & b_{s-1} & b_s
\end{array}
\]
where $a_{ij}=0$ for $j\ge i$, $c_i=\sum_{j<i}a_{ij}$, and $0\le c_i\le 1$.

Define the stage times and corresponding  evaluations of $F$ by
\[
 t_{k,i}:=t_k+c_ih,
 \qquad F_{k,i}:=F(t_{k,i},Y_{k,i}).
\]
Starting with \(Y_{k,1}:=Y_k\), define the stage increments and the final increment by
\begin{equation}\label{eq:common-base-increments}
 \begin{aligned}
  \Delta A_{k,i}
  &:=h\sum_{j=1}^{i-1}a_{ij}F_{k,j},
  &&2\le i\le s,\\
  \Delta A_k
  &:=h\sum_{j=1}^{s}b_jF_{k,j}.
 \end{aligned}
\end{equation}
For each \(2\le i\le s\), the stage approximation \(Y_{k,i}\) is computed by applying the \(K\)-\(L\) integrator \eqref{eq:reduced-kl} to \((U_k,S_k,V_k)\) with \(\Delta A_{k,i}\). The final approximation $Y_{k+1}$ is computed in the same way, using \(\Delta A_k\).  Algorithm~\ref{alg:common-base-rk} summarizes one step.

\begin{algorithm}[htbp]
\caption{One step of the $K$-$L$ Runge--Kutta method.}
\label{alg:common-base-rk}
\begin{algorithmic}[1]
\Require \(t_k,h\), the factors $(U_k,S_k,V_k)$, and
an explicit $s$-stage tableau \((a_{ij},b_i,c_i)\), $s\ge1$
\State Set \(Y_{k,1}:=Y_k\) and \(F_{k,1}:=F(t_k,Y_k)\).
\For{\(i=2,\ldots,s\)}
 \State Set \(\Delta A_{k,i}:=h\sum_{j=1}^{i-1}a_{ij}F_{k,j}\).
 \State Apply the \(K\)-\(L\) integrator \eqref{eq:reduced-kl}
 to \((U_k,S_k,V_k)\) with increment \(\Delta A_{k,i}\) to obtain \(Y_{k,i}\).
 \State Set \(F_{k,i}:=F(t_{k,i},Y_{k,i})\).
\EndFor
\State Set \(\Delta A_k:=h\sum_{i=1}^{s}b_iF_{k,i}\).
\State Apply the \(K\)-\(L\) integrator \eqref{eq:reduced-kl} to \((U_k,S_k,V_k)\) with increment
\(\Delta A_k\).
\Ensure The resulting approximation \(Y_{k+1}=U_{k+1}S_{k+1}V_{k+1}^\top\).
\end{algorithmic}
\end{algorithm}

% For \(s=1\), \(c_1=0\), and \(b_1=1\), the stage loop is empty, so Algorithm~\ref{alg:common-base-rk} performs a single \(K\)-\(L\) step with the forward Euler increment \(\Delta A_k=hF(t_k,Y_k)\).

\subsection{Conditions for robust high-order convergence}

We give sufficient conditions on the Runge--Kutta coefficients for robust high-order convergence. The conditions control the quadrature accuracy of each stage, the propagation of errors from earlier stages, and the accuracy of the stages used in the final increment. The resulting stage estimates yield the required increment accuracy and, consequently, the local and global error bounds.

%For each stage \(i\ge2\), we assign an integer \(q_i\) to obtain a stage error of \(O(h^{q_i+1}+h\eps_r)\). 

Define
\begin{equation}\label{eq:flow-field-defined}
 g_k(t):=F(t,A_k(t)),\qquad t_k\le t\le t_{k+1},
\end{equation}
with \(A_k\) defined in \eqref{eq:increment-local-solution}. For an integer \(p\geq 1\), assume that, at every step
\begin{equation}\label{eq:assumption-smoothness}
 g_k\in C^p([t_k,t_{k+1}]),\qquad
 \sup_{t_k\le t\le t_{k+1}}\|g_k^{(l)}(t)\|\le M_l,
 \quad 1\le l\le p,
\end{equation}
where the constants \(M_1,\ldots,M_p\) are independent of \(k\), the step size \(0<h\le\min(1,h_*)\), and the singular values.

\begin{theorem}
\label{thm:rk-stage-criterion}
Under the standing assumptions and the regularity assumption \eqref{eq:assumption-smoothness}, consider Algorithm~\ref{alg:common-base-rk} with an explicit Runge--Kutta tableau satisfying \(0<c_i\le1\) for \(2\le i\le s\) and
\begin{equation}\label{eq:rk-quadrature-moments}
 \sum_{i=1}^{s}b_i=1,\qquad
 \sum_{i=1}^{s}b_i c_i^l=\frac1{l+1}
 \quad(1\le l\le p-1).
\end{equation}
Suppose either $p=s=1$, or $p\geq 2$ and there exist integers \(q_i\in\{1,\ldots,p-1\}\), \(2\le i\le s\), such that
\begin{align}
 \sum_{j<i}a_{ij}c_j^l&=\frac{c_i^{l+1}}{l+1},
 &&1\le l\le q_i-1,\quad 2\le i\le s,
 \label{eq:rk-row-moments}\\
 a_{ij}\ne0&\ \Longrightarrow\ q_j\ge q_i-1,
 &&2\le j<i\le s,
 \label{eq:rk-stage-dependencies}\\
 b_i\ne0&\ \Longrightarrow\ q_i=p-1,
 &&2\le i\le s.
 \label{eq:rk-output-support}
\end{align}
%For \(p=s=1\), the conditions involving \(q_i\) are void.
Then, for \(0<h\le h_0:=\min(1,h_*)\), the stage approximations satisfy, when $s\geq 2$,
\begin{equation}\label{eq:rk-assigned-stage-errors}
 \|Y_{k,i}-A_k(t_{k,i})\|=O\bigl(h^{q_i+1}+h\eps_r\bigr),
 \qquad 2\le i\le s.
\end{equation}
In both cases, the final increment and the numerical solution satisfy
\begin{equation}\label{eq:rk-criterion-local}
 d_k=O(h^{p+1}+h^2\eps_r),\qquad
 \|Y_{k+1}-A_k(t_{k+1})\|=O(h^{p+1}+h\eps_r).
\end{equation}
Moreover, the exact solution \(A(t)\) of \eqref{eq:full-ode} with \(\|Y_0-A(t_0)\|\le\delta\) satisfies the global error bound
\begin{equation} \label{eq:global-bound}
 \max_{t_0\leq t_k\le T}\|Y_k-A(t_k)\|=O(\delta+\eps_r+h^p).
\end{equation}
Thus, the method is robust of order \(p\). The constants depend only on \(r,L_F,T,p,\min(m,n)\), the tableau, and \(M_1,\ldots,M_p\).
\end{theorem}

\begin{proof}
\emph{Stage errors.}
Fix a step \(k\), and set \(\eta_i:=\|Y_{k,i}-A_k(t_{k,i})\|\) for \(1\le i\le s\).  Since \(c_1=0\) and \(Y_{k,1}=Y_k=A_k(t_k)\), we have \(\eta_1=0\). For $s\geq 2$, we first prove \eqref{eq:rk-assigned-stage-errors}. 

For each \(2\le i\le s\),  using \(A_k(t_{k,i})=Y_k+\int_{t_k}^{t_{k,i}}g_k(t)\,dt\) and \(\Delta A_{k,i}\) from \eqref{eq:common-base-increments}, we obtain
\begin{align}
 Y_k+\Delta A_{k,i}-A_k(t_{k,i})
 &=h\sum_{j<i}a_{ij}\bigl(F_{k,j}-g_k(t_{k,j})\bigr) \notag\\
 &\quad+h\sum_{j<i}a_{ij}\,g_k(t_{k,j})-\int_{t_k}^{t_{k,i}}g_k(t)\,dt. \label{eq:stage-error}
\end{align}
The second line is the quadrature error of stage \(i\). To estimate this term, Taylor's theorem and the regularity assumption \eqref{eq:assumption-smoothness} give, for any integer $1\leq q\leq p$,
\[
 g_k(t)=\sum_{l=0}^{q-1}\frac{(t-t_k)^l}{l!}\,g_k^{(l)}(t_k)
 +R_{k,q}(t),
 \qquad
 \|R_{k,q}(t)\|\le\frac{M_q}{q!}\,(t-t_k)^q, \quad t\in [t_k,t_{k+1}].
\]
Applying the expansion with $q=q_i$, we obtain
\begin{align*}
 &h\sum_{j<i} a_{ij}  g_k(t_{k,j})
   -\int_{t_k}^{t_{k,i}}g_k(t)\,dt\\
 =&\sum_{l=0}^{q_i-1}\frac{h^{l+1}}{l!}\,g_k^{(l)}(t_k)
   \left(\sum_{j<i} a_{ij} c_j^l-\frac{c_i^{l+1}}{l+1}\right)
   +h\sum_{j<i} a_{ij} R_{k,q_i}(t_{k,j})
   -\int_{t_k}^{t_{k,i}}R_{k,q_i}(t)\,dt.
\end{align*}
The polynomial terms vanish by \(c_i=\sum_{j<i}a_{ij}\) and \eqref{eq:rk-row-moments}. Consequently,
\[
 \left\|h\sum_{j<i}a_{ij}\,g_k(t_{k,j})-\int_{t_k}^{t_{k,i}}g_k(t)\,dt\right\|
 =O(h^{q_i+1}).
\]
Since 
\[
\left\|F_{k,j}-g_k(t_{k,j})\|=\|F(t_{k,j},Y_{k,j})-F(t_{k,j},A_k(t_{k,j}))\right\|\leq L_F\eta_j,
\]
\eqref{eq:stage-error} implies
\[
\left\|Y_k+\Delta A_{k,i}-A_k(t_{k,i})\right\|\leq hL_F \sum_{j<i} |a_{ij}|\eta_j +O(h^{q_i+1}).
\]
The stage approximation \(Y_{k,i}\) is computed by applying the \(K\)-\(L\) integrator to \((U_k,S_k,V_k)\) with increment \(\Delta A_{k,i}\). Theorem~\ref{thm:accuracy-transfer}, applied with \(t_k\), \(c_ih\in(0,h_*]\), \(A_k\), \(\Delta A_{k,i}\), and \(Y_{k,i}\) in place of \(t_0\), \(h\), \(A\), \(\Delta A\), and \(Y_1\), gives
\begin{align}
 \eta_i
 &\le\gamma(c_ih)\,\|Y_k+\Delta A_{k,i}-A_k(t_{k,i})\|
   +(1+\gamma(c_ih))\sqrt{\min(m,n)}\,\eps_r\varphi(c_ih)\notag\\
 &\le\gamma(c_ih)\Bigl(hL_F\sum_{j<i}|a_{ij}|\eta_j
   +O(h^{q_i+1})\Bigr)+(1+\gamma(c_i h))\sqrt{\min(m,n)}\,\eps_r\varphi(c_ih)\notag\\
   &\leq \gamma_* h L_F \sum_{j<i} |a_{ij}|\eta_j +O(h^{q_{i}+1}+h\varepsilon_r),
 \label{eq:stage-recursion}
\end{align}
where we used $\gamma(c_ih)\leq \gamma_*$ and $\varphi(c_i h) \leq he ^{L_F}$ for $h\leq h_0$ in the last line.

We prove \eqref{eq:rk-assigned-stage-errors} by induction over \(i\). For $i=2$, the sum in \eqref{eq:stage-recursion} involves only $\eta_1=0$, so $\eta_2=O(h^{q_2+1}+h\varepsilon_r)$. Suppose that \(\eta_j=O(h^{q_j+1}+h\eps_r)\) holds for \(2\le j\le i-1\). If \(j\ge2\) and \(a_{ij}\ne0\), then \(q_j\ge q_i-1\) by
\eqref{eq:rk-stage-dependencies}, and, since \(h\le1\),
\[
h\eta_j=O(h^{q_j+2}+h^2\eps_r)=O(h^{q_i+1}+h\eps_r).
\]
\eqref{eq:stage-recursion} gives \(\eta_i=O(h^{q_i+1}+h\eps_r)\), proving \eqref{eq:rk-assigned-stage-errors}. 

We now establish the stage accuracy needed for the final increment. If \(p=s=1\), the only stage satisfies \(\eta_1=0\). If \(s\ge2\), then for every \(i\ge2\) with \(b_i\ne0\), condition \eqref{eq:rk-output-support} gives \(q_i=p-1\), so \eqref{eq:rk-assigned-stage-errors} implies
\begin{equation}\label{eq:rk-criterion-assumptions}
 \|Y_{k,i}-A_k(t_{k,i})\|=O(h^p+h\eps_r)
 \quad\text{if }b_i\ne0.
\end{equation}

\medskip
\noindent\emph{Increment and solution errors.}
Since \(A_k(t_{k+1})=Y_k+\int_{t_k}^{t_{k+1}}g_k(t)\,dt\),
\begin{align}
 Y_k+\Delta A_k-A_k(t_{k+1})
 &=h\sum_{i=1}^{s}b_i
   \bigl(F_{k,i}-g_k(t_{k,i})\bigr)+h\sum_{i=1}^{s}b_i g_k(t_{k,i})
       -\int_{t_k}^{t_{k+1}}g_k(t)\,dt.
 \label{eq:rk-criterion-decomposition}
\end{align}
Applying the same Taylor argument with $q=p$ and using \eqref{eq:rk-quadrature-moments}, we obtain
\begin{equation}\label{eq:rk-criterion-quadrature}
 \begin{aligned}
 \left\|h\sum_{i=1}^{s}b_i g_k(t_{k,i})
        -\int_{t_k}^{t_{k+1}}g_k(t)\,dt\right\| =O(h^{p+1}).
 \end{aligned}
\end{equation}
Combining \eqref{eq:rk-criterion-assumptions} and \eqref{eq:rk-criterion-quadrature} with the Lipschitz bound,
\eqref{eq:rk-criterion-decomposition} yields
\begin{equation}\label{eq:rk-criterion-defect}
 \begin{aligned}
 d_k=\|Y_k+\Delta A_k-A_k(t_{k+1})\|
 &\le hL_F\sum_{i=1}^{s}|b_i|\eta_i+O(h^{p+1})=O(h^{p+1}+h^2\eps_r).
 \end{aligned}
\end{equation}

Applying Theorem~\ref{thm:accuracy-transfer} on $[t_k,t_{k+1}]$ and using
$\gamma(h)\le\gamma_*$ and $\varphi(h)\le he^{L_F}$
for $h\le h_0$, we obtain
\[
\begin{aligned}
\|Y_{k+1}-A_k(t_{k+1})\|
&\le \gamma(h)d_k
 +(1+\gamma(h))\sqrt{\min(m,n)}
 \,\varepsilon_r\varphi(h)=O(h^{p+1}+h\varepsilon_r).
\end{aligned}
\]
Together with \eqref{eq:rk-criterion-defect}, this proves \eqref{eq:rk-criterion-local}.

To obtain the global error estimate, we return to the increment error bound \eqref{eq:rk-criterion-defect}. Since $h\le 1$, it implies
\[
d_h=O(h^{p+1}+h^2\varepsilon_r)
   =O(h^{p+1}+h\varepsilon_r).
\]
The global error bound \eqref{eq:global-bound} then follows directly from Theorem~\ref{thm:general-increments}.
\end{proof}

\subsection{Examples of orders one to four}
\label{sec:rk-examples}
We consider four choices of Runge--Kutta coefficients: forward Euler, the explicit midpoint method, Heun's third-order method, and Merson's fourth-order method~\cite{Merson1957}. The corresponding low-rank methods defined by Algorithm~\ref{alg:common-base-rk} are denoted by Euler, RK2-Midpoint, RK3-Heun, and RK4-Merson, respectively. Their Butcher tableaux are given below:
\begin{equation}\label{eq:example-tableaux}
\begin{array}{c|c}
0&0\\ \hline
&1
\end{array}
\qquad
\begin{array}{c|cc}
0&0&0\\
\frac12&\frac12&0\\ \hline
&0&1
\end{array}
\qquad
\begin{array}{c|ccc}
0&0&0&0\\
\frac13&\frac13&0&0\\
\frac23&0&\frac23&0\\ \hline
&\frac14&0&\frac34
\end{array}
\qquad
\begin{array}{c|ccccc}
0&0&0&0&0&0\\
\frac13&\frac13&0&0&0&0\\
\frac13&\frac16&\frac16&0&0&0\\
\frac12&\frac18&0&\frac38&0&0\\
1&\frac12&0&-\frac32&2&0\\ \hline
&\frac16&0&0&\frac23&\frac16
\end{array}.
\end{equation}

\begin{corollary}
\label{cor:examples}
Under the standing assumptions and \eqref{eq:assumption-smoothness} with \(p=1,2,3,4\), respectively, Euler, RK2-Midpoint, RK3-Heun, and RK4-Merson in Algorithm~\ref{alg:common-base-rk} with \(0<h\le h_0\) satisfy the local and global error bounds of Theorem~\ref{thm:rk-stage-criterion}. In particular, they are robust of orders one, two, three, and four.
\end{corollary}

\begin{proof}
We verify the conditions of Theorem~\ref{thm:rk-stage-criterion} for each scheme.

For Euler, take \(p=s=1\), \(c_1=0\), and \(b_1=1\). The stage conditions are void, and \eqref{eq:rk-quadrature-moments} reduces to \(b_1=1\). Thus Theorem~\ref{thm:rk-stage-criterion} applies with \(\Delta A_k=hF(t_k,Y_k)\).

For RK2-Midpoint, take \(q_2=1=p-1\).  Then \eqref{eq:rk-row-moments} is void because \(q_2-1=0\), \eqref{eq:rk-stage-dependencies} is void because stage~\(2\) uses only stage~\(1\), and \eqref{eq:rk-output-support} holds. The weights give \(b_1+b_2=1\) and \(b_2c_2=\frac12\), which is \eqref{eq:rk-quadrature-moments} with \(p=2\). In fact, the same argument applies to any explicit two-stage Runge--Kutta method of order two with $0<c_2\leq 1$.

For RK3-Heun, take \((q_2,q_3)=(1,2)\). Row~\(3\), with \(a_{32}=\frac23\) and \(c_2=\frac13\), is the midpoint rule on \([0,\frac23]\): \(a_{32}c_2=\frac29=\frac12c_3^2\), which is
\eqref{eq:rk-row-moments}.  Stage~\(3\) uses only stage~\(2\), and \(q_2=q_3-1\), which is \eqref{eq:rk-stage-dependencies}.  The weights give \(\sum_ib_i=1\), \(\sum_ib_ic_i=\frac34\cdot\frac23=\frac12\), and \(\sum_ib_ic_i^2=\frac34\cdot\frac49=\frac13\), which is \eqref{eq:rk-quadrature-moments} with \(p=3\), and
\eqref{eq:rk-output-support} holds because \(b_2=0\) and \(q_3=2=p-1\).

For RK4-Merson, take \((q_2,q_3,q_4,q_5)=(1,2,3,3)\).  Row~\(3\) is the trapezoidal rule on \([0,\frac13]\): \(a_{32}c_2=\frac1{18}=\frac12c_3^2\). Rows~\(4\) and~\(5\) integrate polynomials of degree two exactly: for \(l=1,2\), the sums \(\sum_{j<4}a_{4j}c_j^l\) equal \(\frac18,\frac1{24}\), that is, \((\frac12)^{l+1}/(l+1)\), and the sums \(\sum_{j<5}a_{5j}c_j^l\) equal \(\frac12,\frac13\).  This gives \eqref{eq:rk-row-moments}. Row~\(3\) uses stage~\(2\), and \(q_2=q_3-1\).  Since \(a_{42}=a_{52}=0\), rows~\(4\) and~\(5\) use only stages \(1\), \(3\), and \(4\), and \(q_3=2\)
and \(q_4=3\) are at least \(q_4-1=q_5-1=2\).  This gives \eqref{eq:rk-stage-dependencies}.  The weights give \(\sum_ib_i=1\) and \(\sum_ib_ic_i^l=\frac23\bigl(\frac12\bigr)^l+\frac16=\frac1{l+1}\) for \(l=1,2,3\), which is \eqref{eq:rk-quadrature-moments} with \(p=4\), and \eqref{eq:rk-output-support} holds because \(b_2=b_3=0\) and \(q_4=q_5=3=p-1\).
\end{proof}

\begin{remark}
\label{rem:rk4-classical}
When Algorithm~\ref{alg:common-base-rk} is applied with the classical four-stage Runge–Kutta method (RK4-Classical) \cite{HairerNorsettWanner1993}, the sufficient conditions of Theorem~\ref{thm:rk-stage-criterion} for robust order four are not satisfied. The numerical experiments in Section~\ref{sec:numerical-sigma} show that the resulting low-rank method can exhibit order reduction in the presence of small singular values.
\end{remark}

%%%%%%%%%%%%%%%%%%%%%%%%
\section{Numerical experiments}
\label{sec:numerical}

We report the relative Frobenius error \(\|Y_N-A(T)\|/\|A(T)\|\) at the final time \(T=t_0+Nh\). In each run, the numerical and reference solutions start from the same rank-\(r\) matrix. The reference solution of the full matrix equation is computed with DOP853 \cite{HairerNorsettWanner1993}.

\subsection{Convergence test for the Allen--Cahn equation}
\label{sec:numerical-allen-cahn}

We test the temporal convergence orders on a matrix Allen--Cahn problem adapted from \cite{NobileRiffaud2026},
\begin{equation}\label{eq:numerical-allen-cahn}
 \dot A=0.01(DA+AD)+A-A^{\circ3},
\end{equation}
where \(A^{\circ3}\) denotes the entrywise cube. Here \(m=n=64\), \(x_i=(i-1)\Delta x\), \(\Delta x=2\pi/63\), and \(D=(\Delta x)^{-2}\operatorname{tridiag}(1,-2,1)\). For \(r=4,8,12,16,24\), the common initial value \(Y_0=A(0)\) is the best rank-\(r\) approximation of the matrix with entries
\begin{equation}\label{eq:numerical-allen-initial}
 \frac{(e^{-\tan^2x_i}+e^{-\tan^2x_j})\sin x_i\sin x_j}
 {1+e^{|\csc(x_i/2)|}+e^{|\csc(x_j/2)|}},
 \qquad 1\le i,j\le64,
\end{equation}
with values at singular arguments defined by continuity. We integrate to \(T=10\).

Figure~\ref{fig:allen-cahn}(a) compares Euler, RK2-Midpoint, RK3-Heun, and RK4-Merson at \(r=24\). The curves show the predicted first-, second-, third-, and fourth-order convergence, respectively.

Figure~\ref{fig:allen-cahn}(b) uses RK3-Heun as a representative case to illustrate the interaction between rank and temporal accuracy. At smaller ranks, the error eventually plateaus as the time step is refined. For sufficiently large ranks, the normal-defect contribution is small enough for the expected temporal order to become visible. In particular, the curve for \(r=24\) exhibits third-order convergence.

\begin{figure}[htbp!]
 \centering
 \includegraphics[width=\textwidth]{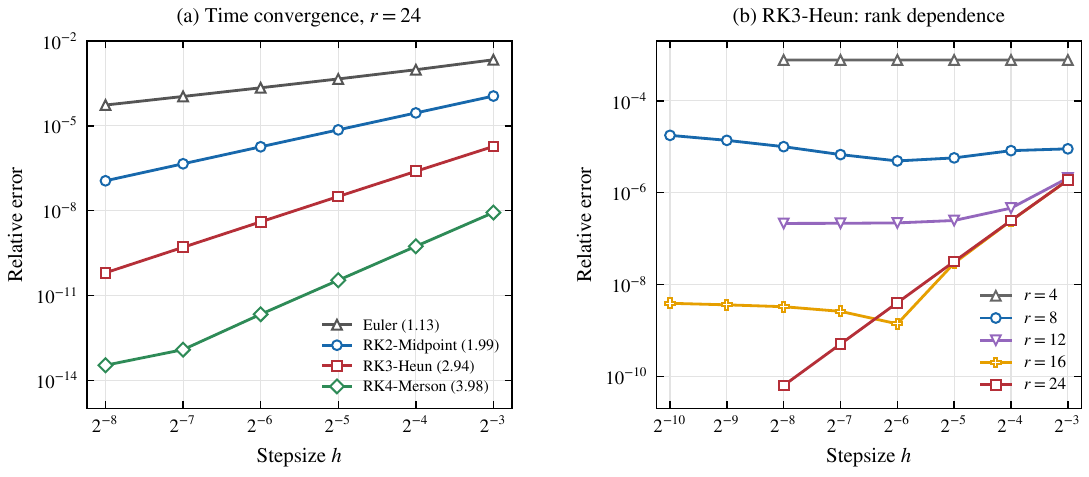}
 \caption{Relative errors for the Allen--Cahn equation at \(T=10\). (a) Time convergence of the four methods at \(r=24\). (b) Time convergence of RK3-Heun at different ranks. Third-order decay becomes visible when the rank is large enough that the normal-defect contribution no longer dominates the error.}
 \label{fig:allen-cahn}
\end{figure}

\subsection{Robustness to small singular values}
\label{sec:numerical-sigma}

To test robustness with respect to small singular values, we consider the linear matrix equation
\begin{equation}\label{eq:numerical-linear}
 \dot A=\Omega(t)A+A\Gamma(t),\qquad
 \Omega(t):=\Omega_0+\sin(2t)\Omega_1,\qquad
 \Gamma(t):=\Gamma_0+\cos(3t)\Gamma_1,
\end{equation}
with \(m=n=20\) and \(r=4\). The matrices \(\Omega_0,\Omega_1,\Gamma_0,\Gamma_1\) are fixed random skew-symmetric matrices, scaled so that \(\|\Omega_0\|_2=0.65\), \(\|\Omega_1\|_2=0.45\), \(\|\Gamma_0\|_2=0.60\), and \(\|\Gamma_1\|_2=0.40\). This equation preserves singular values, and \(F(t,X)=\Omega(t)X+X\Gamma(t)\) satisfies \((I-XX^+)F(t,X)(I-X^+X)=0\) for every \(X\). Hence \(\eps_r=0\). For fixed random matrices \(U_0,V_0\in\R^{20\times4}\) with orthonormal columns, set
\begin{equation}\label{eq:numerical-sigma-initial}
 Y_0=A(0)=U_0\operatorname{diag}(1,10^{-1},10^{-2},\sigma)V_0^\top,
 \qquad \sigma\in\{10^{-2},10^{-6},10^{-10},10^{-14}\}.
\end{equation}
We integrate to \(T=1\). 

Figure~\ref{fig:small-sigma} shows first-, second-, third-, and fourth-order convergence for Euler, RK2-Midpoint, RK3-Heun, and RK4-Merson, respectively, across all four values of \(\sigma\). RK4-Classical shows fourth-order convergence for \(\sigma=10^{-2}\), but irregular convergence for \(\sigma=10^{-6}\). For \(\sigma=10^{-10}\), the error decreases with a fitted slope of \(3.01\) over \(2^{-7}\le h\le2^{-3}\), while further refinement no longer reduces the error consistently. For \(\sigma=10^{-14}\), the observed convergence is approximately third order. Over the tested range of step sizes, RK4-Classical exhibits fourth-order decay for \(\sigma=10^{-2}\), but this behavior is not consistent for the smaller values of \(\sigma\).

\begin{figure}[htbp!]
 \centering
 \includegraphics[width=\textwidth]{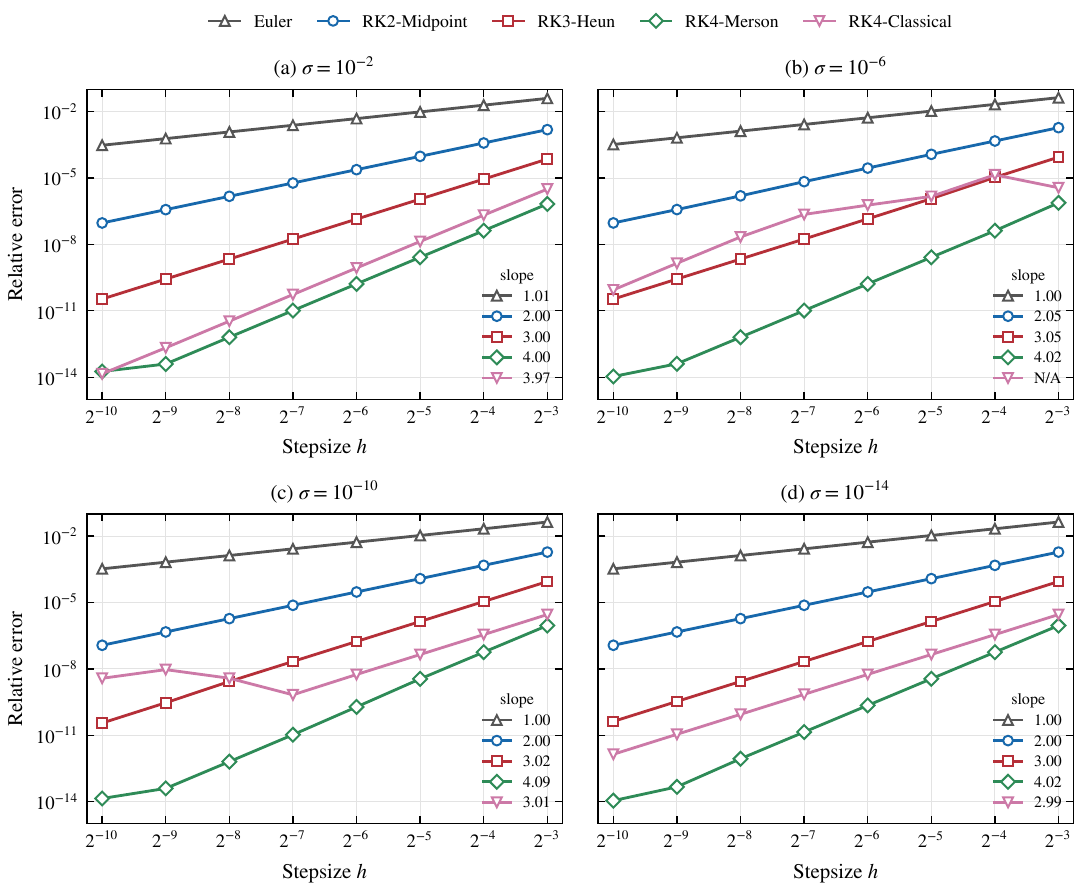}
 \caption{Relative errors for the linear matrix equation at \(T=1\), with one panel for each value of the smallest singular value \(\sigma\). The legends report the fitted slopes. In (b), N/A indicates that no single slope describes the irregular convergence of RK4-Classical.}
 \label{fig:small-sigma}
\end{figure}

%%%%%%%%%%%%%%%%%%%%%%%
\section{Conclusion}
\label{sec:conclusion}

We have developed and analyzed robust high-order integrators for dynamical low-rank approximation based on an increment-based $K$-$L$ formulation of projector splitting. This formulation eliminates the separate backward $S$-step and is algebraically equivalent to the fixed-increment $K$-$S$-$L$ integrator. The key analytical ingredient is a relaxed-exactness estimate, which controls the error of the $K$-$L$ integrator in terms of the prescribed increment error. Combined with a comparison-matrix construction, this yields robust local and global error estimates. In particular, sufficiently accurate increments lead to the global error bound
\begin{equation*}
 \max_{t_0\leq t_k\le T}\|Y_k-A(t_k)\|
 \le C\bigl(\delta+\eps_r+h^p\bigr),
\end{equation*}
where $\delta$ is the initial error, $\eps_r$ bounds the normal component of the right-hand side, $h$ is the time step, and $p$ is the convergence order. The constant $C$ and the admissible step-size threshold are independent of small singular values.

Building on this analysis, we have constructed low-rank Runge–Kutta methods in which the $K$-$L$ integrator is used to compute both the stage and final approximations, without basis augmentation or rank truncation. We derived sufficient conditions on the Runge–Kutta coefficients that ensure the required stage accuracy and hence robust high-order convergence. The resulting methods based on the explicit midpoint method, Heun’s third-order method, and Merson’s fourth-order method achieve robust orders two, three, and four, respectively. Numerical experiments confirm the predicted convergence orders and illustrate their robustness with respect to small singular values.

%%%%%%%%%%%%%%%%%%%%%%%%%
\appendix
\makeatletter
\renewcommand\@seccntformat[1]{%
 \@ifundefined{#1@cntformat}{\csname the#1\endcsname\quad}%
  {\csname #1@cntformat\endcsname}}
\newcommand\section@cntformat{\appendixname~\thesection\quad}
\makeatother
\section{Construction of a comparison matrix}
\label{app:comparison-matrix}

We prove Lemma~\ref{lem:comparison-matrix} for continuous \(F\) satisfying \eqref{eq:assumption-BL} and \eqref{eq:assumption-normal}.  Lemma~\ref{lem:endpoint-one-sided}
decomposes the right-hand side $F(t,X)$ at a matrix \(X=KV^\top\) with \(V^\top V=I_r\). The matrix \(H\) in this decomposition is defined in \eqref{eq:endpoint-H-construction} through a singular value decomposition (SVD) of \(K\), which is not unique when singular values coincide, so \(H\) need not depend continuously on the state.  We therefore take the
comparison matrix as a subsequential limit of the final iterates \(\bar X_\nu\) of the recursion \eqref{eq:endpoint-auxiliary-scheme} with \(\nu\) steps of size \(h/\nu\).

\subsection{A decomposition of the right-hand side $F(t,X)$}

\begin{lemma}
\label{lem:endpoint-one-sided}
Let \(X=KV^\top\), where \(K\in\R^{m\times r}\) and \(V\in\R^{n\times r}\) satisfies \(V^\top V=I_r\). For every \(t\in[0,T]\), there are matrices \(H\in\R^{n\times r}\) and \(E\in\R^{m\times n}\) such that
\begin{equation}
\label{eq:endpoint-one-sided-decomposition}
F(t,X)=F(t,X)VV^\top+KH^\top+E,
\end{equation}
with
\begin{equation}
\label{eq:endpoint-one-sided-bounds}
V^\top H=0,
\qquad
\|H\|\le\sqrt r\,L_F,
\qquad
\|E\|\le\sqrt{\min(m,n)}\,\eps_r.
\end{equation}
\end{lemma}

\begin{proof}
Take a compact SVD
\[
K=\sum_{i=1}^{\rank(K)}\sigma_i u_i w_i^\top,
\qquad \sigma_i>0,
\]
with orthonormal \(u_i\) and orthonormal \(w_i\).  Set \(v_i=Vw_i\) and \(X_i=(K-\sigma_iu_iw_i^\top)V^\top=X-\sigma_i u_i v_i^\top\).  Then \(X=\sum_i\sigma_iu_iv_i^\top\) with orthonormal \(v_i\), so that
\[
XX^+=\sum_{i=1}^{\rank(K)}u_iu_i^\top,
\qquad
u_i^\top X_i=0,
\qquad
X_i=X_iVV^\top,
\qquad
\|X-X_i\|=\sigma_i.
\]
Define
\begin{equation}
\label{eq:endpoint-H-construction}
H:=\sum_{i=1}^{\rank(K)}
\frac{(I-VV^\top)\bigl(F(t,X)-F(t,X_i)\bigr)^\top u_i\,w_i^\top}{\sigma_i}.
\end{equation}
All sums are empty when \(K=0\).  Then \(V^\top H=0\), and \eqref{eq:assumption-BL} gives
\[
\|Hw_i\|_2\le\frac{\|F(t,X)-F(t,X_i)\|}{\sigma_i}\le L_F,
\qquad
\|H\|^2=\sum_{i=1}^{\rank(K)}\|Hw_i\|_2^2\le\rank(K)\,L_F^2.
\]
Since \(KH^\top=\sum_iu_iu_i^\top\bigl(F(t,X)-F(t,X_i)\bigr)(I-VV^\top)\), the matrix \(E\) defined by \eqref{eq:endpoint-one-sided-decomposition} is
\begin{equation}
\label{eq:endpoint-E-decomposition}
E=(I-XX^+)F(t,X)(I-VV^\top)
  +\sum_{i=1}^{\rank(K)} u_iu_i^\top F(t,X_i)(I-VV^\top).
\end{equation}
These \(\rank(K)+1\) terms are orthogonal in the Frobenius inner product, because \((I-XX^+)u_i=0\) and the \(u_i\) are orthonormal.  Since
\(X=XVV^\top\) and \(X_i=X_iVV^\top\), we have \(X^+XVV^\top=X^+X\) and
\(X_i^+X_iVV^\top=X_i^+X_i\), so
\[
I-VV^\top=(I-X^+X)(I-VV^\top)=(I-X_i^+X_i)(I-VV^\top).
\]
Hence the first term of \eqref{eq:endpoint-E-decomposition} equals \((I-XX^+)F(t,X)(I-X^+X)(I-VV^\top)\).  By \eqref{eq:assumption-normal}, its spectral norm is at most \(\eps_r\), and its rank is at most \(\min(m-\rank(K),n-r)\) because \(\rank(X)=\rank(K)\). Since \(u_i^\top X_i=0\), we have \(u_i^\top=u_i^\top(I-X_iX_i^+)\), and \eqref{eq:assumption-normal} at \(X_i\in\Mle\) gives
\(\|u_i^\top F(t,X_i)(I-VV^\top)\|_2\le\eps_r\). With \(\|Z\|^2\le\rank(Z)\|Z\|_2^2\) for the first term,
\begin{align*}
\|E\|^2
&=\|(I-XX^+)F(t,X)(I-VV^\top)\|^2
 +\sum_{i=1}^{\rank(K)}\|u_i^\top F(t,X_i)(I-VV^\top)\|_2^2\\
&\le\bigl(\min(m-\rank(K),n-r)+\rank(K)\bigr)\eps_r^2
\le\min(m,n)\,\eps_r^2.
\end{align*}
\end{proof}

\subsection{Proof of Lemma~\ref{lem:comparison-matrix}}

\begin{proof}
Fix \(h\) and put \(\kappa=\sqrt r\,L_F\). For each positive integer \(\nu\), let \(\tau=h/\nu\) and \(\bar t_j=t_0+j\tau\) for \(0\le j\le \nu\), so that \(\bar t_0=t_0\) and \(\bar t_\nu=t_1\).  The auxiliary quantities below depend on \(\nu\), which we suppress in the notation. Initialize \(\bar V_0=V_0\), \(\bar K_0=Y_0V_0\), and \(\bar X_0=\bar K_0\bar V_0^\top=Y_0\).
For \(j=0,\ldots,\nu-1\), set \(F_j=F(\bar t_j,\bar X_j)\) and choose \(H_j,E_j\) from Lemma~\ref{lem:endpoint-one-sided} at \((\bar t_j,\bar K_j,\bar V_j)\). Define
\begin{equation}
\label{eq:endpoint-auxiliary-scheme}
\begin{aligned}
\Theta_j&=H_j\bar V_j^\top-\bar V_jH_j^\top,\\
\bar K_{j+1}&=\bar K_j+\tau F_j\bar V_j,\\
\bar V_{j+1}&=e^{\tau\Theta_j}\bar V_j,
\qquad \bar X_{j+1}=\bar K_{j+1}\bar V_{j+1}^\top.
\end{aligned}
\end{equation}

\emph{Uniform bounds.}
Since \(\Theta_j^\top=-\Theta_j\), the matrix \(e^{s\Theta_j}\) is orthogonal for every \(s\in\R\).  By induction, \(\bar V_j^\top\bar V_j=I_r\) and \(\rank(\bar X_j)\le r\) for \(0\le j\le\nu\).  Let
\(W_j\in\R^{n\times(n-r)}\) be such that \([\bar V_j,W_j]\) is orthogonal. For \(r=n\), \(W_j\) has no columns and \(H_j=\Theta_j=0\), so the identities below also hold. Since \(\bar V_j^\top H_j=0\), we have \(H_j=W_jW_j^\top H_j\),
\[
\Theta_j\bar V_j=H_j\bar V_j^\top\bar V_j-\bar V_jH_j^\top\bar V_j=H_j,
\qquad
[\bar V_j,W_j]^\top\Theta_j[\bar V_j,W_j]
=\begin{pmatrix}0&-H_j^\top W_j\\W_j^\top H_j&0\end{pmatrix},
\]
and hence \(\|\Theta_j\|_2=\|W_j^\top H_j\|_2=\|H_j\|_2\le\kappa\). Since \(\|F_j\bar V_j\|\le\|F_j\|\le B\),
\begin{equation}
\label{eq:endpoint-K-bound}
\|\bar K_j\|\le\|Y_0\|+B(\bar t_j-t_0)\le\|Y_0\|+Bh.
\end{equation}
Since \(\frac{d}{ds}e^{s\Theta_j}=e^{s\Theta_j}\Theta_j\) and \(\Theta_j\bar V_j=H_j\), \eqref{eq:endpoint-auxiliary-scheme} gives
\[
\begin{aligned}
\bar V_{j+1}-\bar V_j
&=\bigl(e^{\tau\Theta_j}-I\bigr)\bar V_j
 =\int_0^\tau e^{s\Theta_j}H_j\,ds,\\
\bar V_{j+1}-\bar V_j-\tau H_j
&=\int_0^\tau\bigl(e^{s\Theta_j}-I\bigr)H_j\,ds.
\end{aligned}
\]
For \(s\ge0\), \(\|e^{s\Theta_j}\|_2=1\), and \(e^{s\Theta_j}-I=\int_0^se^{u\Theta_j}\Theta_j\,du\) gives \(\|e^{s\Theta_j}-I\|_2\le s\|\Theta_j\|_2\le s\kappa\). With \(\|H_j\|_2\le\kappa\), it follows that
\begin{equation}
\label{eq:endpoint-frame-increments}
\begin{aligned}
\|\bar V_{j+1}-\bar V_j\|_2
&\le\int_0^\tau\|e^{s\Theta_j}\|_2\,\|H_j\|_2\,ds\le\kappa\tau,\\
\|\bar V_{j+1}-\bar V_j-\tau H_j\|_2
&\le\int_0^\tau\|e^{s\Theta_j}-I\|_2\,\|H_j\|_2\,ds
 \le\int_0^\tau\kappa^2s\,ds=\tfrac12\kappa^2\tau^2.
\end{aligned}
\end{equation}
Summing the first bound over \(j=0,\ldots,\nu-1\) gives
\begin{equation}
\label{eq:endpoint-frame-drift}
\|\bar V_\nu-V_0\|_2\le\kappa h.
\end{equation}

\emph{Comparison with the exact solution.}
By \eqref{eq:endpoint-auxiliary-scheme} and \(\bar X_j=\bar K_j\bar V_j^\top\),
\[
\begin{aligned}
\bar X_{j+1}
&=(\bar K_j+\tau F_j\bar V_j)\bar V_j^\top
 +(\bar K_j+\tau F_j\bar V_j)(\bar V_{j+1}-\bar V_j)^\top\\
&=\bar X_j+\tau F_j\bar V_j\bar V_j^\top+\tau\bar K_jH_j^\top
 +\bar K_j(\bar V_{j+1}-\bar V_j-\tau H_j)^\top
 +\tau F_j\bar V_j(\bar V_{j+1}-\bar V_j)^\top\\
&=\bar X_j+\tau\bigl(F_j\bar V_j\bar V_j^\top+\bar K_jH_j^\top\bigr)
 +\mathcal R_j,
\end{aligned}
\]
where
\[
\mathcal R_j
:=\bar K_j(\bar V_{j+1}-\bar V_j-\tau H_j)^\top
 +\tau F_j\bar V_j(\bar V_{j+1}-\bar V_j)^\top.
\]
By \eqref{eq:endpoint-one-sided-decomposition}, \(F_j\bar V_j\bar V_j^\top+\bar K_jH_j^\top=F_j-E_j\), and \(\|E_j\|\le\sqrt{\min(m,n)}\,\eps_r\) by \eqref{eq:endpoint-one-sided-bounds}. Since
\(\|ZW^\top\|\le\|Z\|\,\|W\|_2\) and \(\|F_j\bar V_j\|\le B\), the bounds \eqref{eq:endpoint-K-bound} and \eqref{eq:endpoint-frame-increments} give
\[
\begin{aligned}
\|\mathcal R_j\|
&\le\|\bar K_j\|\,\|\bar V_{j+1}-\bar V_j-\tau H_j\|_2
 +\tau\|F_j\bar V_j\|\,\|\bar V_{j+1}-\bar V_j\|_2\\
&\le(\|Y_0\|+Bh)\,\tfrac12\kappa^2\tau^2+\tau B\,\kappa\tau
 =\bigl((\|Y_0\|+Bh)\kappa^2/2+B\kappa\bigr)\tau^2.
\end{aligned}
\]
Hence
\begin{equation}
\label{eq:endpoint-perturbed-euler}
\begin{aligned}
\bar X_{j+1}&=\bar X_j+\tau F(\bar t_j,\bar X_j)-\tau E_j+\mathcal R_j,\\
\|E_j\|&\le\sqrt{\min(m,n)}\,\eps_r,
\qquad \|\mathcal R_j\|\le C_0\tau^2,
\end{aligned}
\end{equation}
where \(C_0=(\|Y_0\|+Bh)\kappa^2/2+B\kappa\) does not depend on \(\nu\).

Let \(g(t)=F(t,A(t))\) and
\[
\omega_A(\tau):=\sup\bigl\{\|g(t)-g(t')\|:\ t,t'\in[t_0,t_1],\
|t-t'|\le\tau\bigr\}.
\]
Since \(g\) is uniformly continuous on \([t_0,t_1]\), \(\omega_A(\tau)\to0\) as \(\tau\to0\), and
\[
\|A(\bar t_{j+1})-A(\bar t_j)-\tau g(\bar t_j)\|
=\Bigl\|\int_{\bar t_j}^{\bar t_{j+1}}\bigl(g(t)-g(\bar t_j)\bigr)\,dt\Bigr\|
\le\tau\,\omega_A(\tau).
\]
Let \(\bar e_j=\|\bar X_j-A(\bar t_j)\|\), so that \(\bar e_0=0\) because \(\bar X_0=Y_0=A(t_0)\).  Since \(g(\bar t_j)=F(\bar t_j,A(\bar t_j))\), \eqref{eq:endpoint-perturbed-euler} gives
\[
\begin{aligned}
\bar X_{j+1}-A(\bar t_{j+1})
&=\bar X_j-A(\bar t_j)
 +\tau\bigl(F(\bar t_j,\bar X_j)-F(\bar t_j,A(\bar t_j))\bigr)
 -\tau E_j+\mathcal R_j\\
&\quad-\bigl(A(\bar t_{j+1})-A(\bar t_j)-\tau g(\bar t_j)\bigr).
\end{aligned}
\]
By the Lipschitz bound in \eqref{eq:assumption-BL}, the bounds in \eqref{eq:endpoint-perturbed-euler}, and the bound on \(A(\bar t_{j+1})-A(\bar t_j)-\tau g(\bar t_j)\) above,
\[
\bar e_{j+1}\le(1+L_F\tau)\bar e_j
 +\tau\bigl(\sqrt{\min(m,n)}\,\eps_r+C_0\tau+\omega_A(\tau)\bigr).
\]
Since \(\bar e_0=0\), \(1+L_F\tau\le e^{L_F\tau}\), and \(\nu\tau=h\), it follows that
\begin{equation}
\label{eq:endpoint-discrete-error}
\begin{aligned}
\bar e_\nu
&\le\tau\sum_{j=0}^{\nu-1}(1+L_F\tau)^j
 \bigl(\sqrt{\min(m,n)}\,\eps_r+C_0\tau+\omega_A(\tau)\bigr)\\
&=\frac{(1+L_F\tau)^\nu-1}{L_F}
 \bigl(\sqrt{\min(m,n)}\,\eps_r+C_0\tau+\omega_A(\tau)\bigr)\\
&\le\varphi(h)\bigl(\sqrt{\min(m,n)}\,\eps_r+C_0\tau+\omega_A(\tau)\bigr).
\end{aligned}
\end{equation}
Since \(\tau=h/\nu\to0\) as \(\nu\to\infty\),
\begin{equation}
\label{eq:endpoint-limsup-error}
\limsup_{\nu\to\infty}\|\bar X_\nu-A(t_1)\|
\le\sqrt{\min(m,n)}\,\eps_r\varphi(h).
\end{equation}

\emph{The terminal limit and its factorization.}
By \eqref{eq:endpoint-K-bound}, \(\|\bar K_\nu\|\le\|Y_0\|+Bh\), and \(\bar V_\nu^\top\bar V_\nu=I_r\).  Hence a subsequence of \((\bar K_\nu,\bar V_\nu)\) converges to a pair \((\bar K_*,\bar V_*)\) with \(\bar V_*^\top\bar V_*=I_r\).  Set \(\widetilde A_1=\bar K_*\bar V_*^\top\). Then \(\rank(\widetilde A_1)\le r\), and \(\bar X_\nu\to\widetilde A_1\) along the subsequence, so \eqref{eq:endpoint-limsup-error} proves \eqref{eq:comparison-matrix-error}. By \eqref{eq:endpoint-frame-drift},
\[
\|(I-V_0V_0^\top)\bar V_*\|_2
=\|(I-V_0V_0^\top)(\bar V_*-V_0)\|_2\le\kappa h,
\]
and the identity
\[
(\bar V_*^\top V_0)(\bar V_*^\top V_0)^\top
=I_r-\bar V_*^\top(I-V_0V_0^\top)\bar V_*
\]
gives \(\sigma_{\min}(\bar V_*^\top V_0)\ge\sqrt{1-\kappa^2h^2}=1/\gamma(h)>0\). Thus \(Q_h:=(\bar V_*^\top V_0)^{-1}\bar V_*^\top\) is well defined and satisfies
\[
(\widetilde A_1V_0)Q_h=\bar K_*\bar V_*^\top=\widetilde A_1,
\qquad
Q_hV_0=I_r,
\qquad
\|Q_h\|_2\le\gamma(h)\le\gamma_*,
\]
which proves \eqref{eq:comparison-matrix-factorization}. The terms \(C_0\tau\) and \(\omega_A(\tau)\) vanish in the terminal limit, so the final bounds are independent of the singular values of \(Y_0\).
\end{proof}

\section*{Acknowledgement}

This work was partially supported by DOE grant DE-SC0023164, NSF grants DMS-2409858 and IIS-2433957, and AFOSR grants FA9550-24-1-0254 and FA9550-26-1-B161. 

\section*{Declaration of AI assistance}
The authors used ChatGPT to assist with developing the construction in Appendix~\ref{app:comparison-matrix}, polishing the manuscript, and critically checking the mathematical arguments. The authors reviewed and verified all AI-assisted material and take full responsibility for the correctness and integrity of the final manuscript.

\begingroup
\renewcommand{\baselinestretch}{1.1}\selectfont
\bibliographystyle{siamplain}
\bibliography{references}
\endgroup

\end{document}